\documentclass[letterpaper,10pt,conference]{ieeeconf}

\IEEEoverridecommandlockouts
\usepackage{amsmath,mathrsfs, enumerate}
\usepackage{soul}
\usepackage{amssymb}
\usepackage{mathtools}
\usepackage{pstricks,pst-plot,psfrag}
\usepackage[all]{xy}
\usepackage{graphicx,subfigure,xspace,bm}
\usepackage{color}
\usepackage{algorithm}
\usepackage[noend]{algorithmic}
\algsetup{indent=2em}

\usepackage{epstopdf}         
\graphicspath{{./epsfiles/}} 
\usepackage{stackengine}
 
\allowdisplaybreaks
 
\newtheorem{theorem}{Theorem}[section]

\newtheorem{proposition}[theorem]{Proposition}

\newcommand{\realnonnegative}{{\mathbb{R}}_{\ge 0}}

\newcommand{\integerspositive}{\mathbb{Z}_{>0}}

\newcommand{\oprocendsymbol}{\hbox{$\bullet$}}
\newcommand{\oprocend}{\relax\ifmmode\else\unskip\hfill\fi\oprocendsymbol}

\DeclareMathOperator*{\minimize}{minimize}
\DeclareMathOperator*{\maximize}{maximize}

\newcommand{\N}{\mathcal{N}}

\newcommand{\longthmtitle}[1]{\mbox{}\textup{\bf (#1):}}

\newcommand{\G}{{\mathcal{G}}}
\newcommand{\E}{{\mathcal{E}}}

\newcommand{\bx}{\mathbf{x}}
\newcommand{\by}{\mathbf{y}}

\newcommand{\Nsum}{\sum_{i=1}^N}
\newcommand{\tsum}{\sum_{t=1}^n}
\newcommand{\tS}{\tilde{S}}

\newcommand{\tI}{\tilde{I}}

\newcommand{\bR}{\bar{R}}
\newcommand{\bB}{\bar{B}}
\newcommand{\bT}{\bar{T}}

\newcommand{\F}{\mathcal{F}}
\newcommand{\X}{\mathcal{X}}
\newcommand{\Y}{\mathcal{Y}}
\newcommand{\budget}{\mathcal{B}}
\newcommand{\budgetR}{\mathcal{B}_r}
\newcommand{\budgetB}{\mathcal{B}_b}
\newcommand{\avginf}{\tilde{I}_n}

\renewcommand{\theenumi}{(\roman{enumi})}

\renewcommand{\baselinestretch}{0.985}

\begin{document}

\title{Infection-Curing Games over Polya Contagion Networks}

\author{Greg Harrington \qquad Fady Alajaji \qquad Bahman Gharesifard \thanks{The authors are with the Department of Mathematics and Statistics at Queen's University, Kingston, ON, Canada~\texttt{12gth@queensu.ca, fady@mast.queensu.ca, bahman@mast.queensu.ca}. This work was partially supported by the Natural Sciences and Engineering Research Council of Canada.}
}
\maketitle

\begin{abstract}
We investigate infection-curing games on a network epidemics model based on the classical Polya urn scheme that accounts for spatial contagion among neighbouring nodes. We first consider the zero-sum game between competing agents using the cost measure for the average infection in the network. Due to the complexity of this problem we define a game on a proxy measure given by the so-called expected network exposure, and prove the existence of a Nash equilibrium that can be determined numerically using gradient descent algorithms. Finally, a number of simulations are performed on small test networks to provide empirical evidence that a Nash equilibrium exists for games defined on the average network infection.
\end{abstract}


\section{Introduction}\label{section:intro}

The study of epidemics on networks is an active research topic (e.g., see~\cite{DE-JK:10,PVM-JO-RK:09,CN-VMP-GJP:16} and the references therein). Real-life examples include the propagation of burst errors in a wireless communication channel~\cite{FA-TF:94}, of a biological disease through a population~\cite{LK-MA-KD-SK-AO:14}, of malware in computer or smartphone systems~\cite{MG-WG-DT:03}, and the dissemination of rumors~\cite{ER:03} and competing opinions~\cite{EA-LAA:05} in social networks.

In this paper, we examine two player zero-sum games over Polya contagion networks as introduced in~\cite{MH-FA-BG:17, MH-FA-BG:18}. The setup for this model is an adaptation of the classical Polya contagion process~\cite{GP-FE:23, GP-FE:28, GP:30} to allow for spatial interactions between neighbouring nodes. This model is similar to the well known class of Susceptible-Infected-Susceptible (SIS) epidemics models~\cite{DE-JK:10}. Indeed in~\cite{MH-FA-BG:18}, it was empirically shown that the network Polya contagion model can mimic the behaviour of the SIS process. Furthermore, SIS models have been studied in a game-theoretic context in~\cite{ARH-SS:17,PEP-JL-CLB-AN-TB:17} and references therein. We herein establish a similar problem for the network Polya contagion model.

First, we define a two-player game on the average infection rate for a connected network. This setup is an extension of the contagion curing problem that was studied extensively in~\cite{MH-FA-BG:18-2}. In order to simplify the problem of finding optimal control policies, we consider a game on the proxy measure of the expected network exposure. We establish that the expected network exposure is convex as a function of the curing parameters and concave as a function of the infection parameters. We prove that under budget constraints, there exists a Nash equilibrium for the game on the expected network exposure, which can be determined using gradient descent algorithms. Finally, we run a series of simulations on small test networks, to provide empirical evidence that a Nash equilibrium exists for the infection-curing game on the average infection rate, and to support the claim that such an equilibrium policy can be closely approximated using the equilibrium policy for the expected network exposure.

The rest of the paper is organized as follows. Section~\ref{sec:prelim} familiarizes the reader with important preliminary material, including a brief description of the classical Polya contagion process as well as the network Polya contagion process. Section~\ref{sec:game} defines the two-player games that are studied in this paper, and proves important convexity/concavity results as well as the existence of a Nash equilibrium for the expected network exposure game. Section~\ref{sec:simulations} presents illustrative simulation results. Finally, Section~\ref{sec:future} concludes the paper and highlights a few areas for potential future research.

\section{Preliminaries}\label{sec:prelim}

Let $(\Omega,\F,P)$ be a probability space. Consider the stochastic process $\{Z_n\}_{n=1}^{\infty}$, where $Z_n=(Z_{1,n},\ldots,Z_{N,n})$ is a random vector on $\Omega$. We represent by $\{\F_n\}_{n=1}^{\infty}$, the natural filtration on the process $\{Z_n\}_{n=1}^{\infty}$. We write the sequence $(Z_{i,1},\ldots,Z_{i,n})$ as $Z_i^n$, and more generally we use the notation $Z_{i,s}^t=(Z_{i,s},Z_{i,s+1},\ldots,Z_{i,t})$, where $1\leq s<t\leq n$. Further clarification of standard probability concepts can be found in various texts such as~\cite{RA-CD:00,GG-DS:01}.

A Nash equilibrium is a solution for a non-cooperative multiplayer game from which no player may improve its payoff by unilaterally changing its strategy ~\cite{JN:50,JN:51}. In this paper, we consider games with deterministic (pure) strategies. In the context of a two player zero-sum game, denote by $\X$ and $\Y$ the set of allowable strategies for player one and two, respectively. More precisely, consider the case where player one receives a payoff $f(x,y)$, with $x\in\X$ and $y\in\Y$ strategies of player one and two, respectively. Then, $(x^*,y^*)\in\X\times\Y$ is a Nash equilibrium if $\forall x\in\X$, $\forall y\in\Y$: 
\begin{equation*}
f(x,y^*)\leq f(x^*,y^*)\leq f(x^*,y).
\end{equation*}

\subsection{Classical Polya Contagion Process}

We recall the classical Polya contagion process~\cite{GP-FE:23, GP-FE:28, GP:30}. An urn contains $R \in\integerspositive$ red balls and $B \in \integerspositive$ black balls. We use $T=R+B$ to denote the total number of balls in the urn. At each time, $n$, a ball is drawn from the urn and returned with $\Delta>0$ balls of the same colour. From an urn sampling point of view the parameter $\Delta$ is integer-valued; however, we allow them to be real-valued for mathematical convenience (e.g., when determining optimal policies). The random variable $Z_{n}$ is used to indicate the colour of the ball on the $n$th draw, where $Z_{n}=1$ if the drawn ball was red and $Z_{n}=0$ if the drawn ball was black. We can then use $U_{n}$ to denote the proportion of red balls in the urn after the $n$th draw, where
\begin{align*}
	U_n &:= \frac{R + \Delta\tsum Z_{t}}{T + n\Delta} = \frac{\rho_c + \delta_c\tsum Z_{t}}{1+n\delta_c}
\end{align*}
where $\rho_c = \frac{R}{T}$ is the initial proportion of red balls in the urn and $\delta_c = \frac{\Delta}{T}$ is a correlation parameter. The conditional probability of drawing a red ball at time $n$ , given the past history of draws $Z^{n-1} := (Z_{1}, \ldots, Z_{n-1})$, is given by
\begin{align*}
	P(Z_n = 1 \ | \ Z^{n-1} ) &= \frac{R + \Delta\sum_{t=1}^{n-1} Z_{t}}{T + (n-1)\Delta} = U_{n-1}.
\end{align*}

\subsection{Network Polya Contagion Process}

We now review the network Polya contagion process introduced in~\cite{MH-FA-BG:17, MH-FA-BG:18}. Let $\G = (V,\E)$ be a network, where $V = \{1, \ldots, N\}$ is the set of $N \in \integerspositive$ nodes and $\E \subset V\times V$ is the set of edges. Throughout, $\G$ is taken to be undirected, and is assumed to be connected. The set of neighbours to node $i$ is given by $\N_i=\{v\in V:(i,v)\in\E\}$. We then define $\N_i'=\{i\}\cup\N_i$. Each node $i\in V$ is assigned an urn with $R_i \in \integerspositive$ red balls and $B_i \in\integerspositive$ black balls, with the total number of balls at node $i$ being denoted by $T_i = R_i+B_i$.

We next introduce the concept of super urns for node $i$, which contains the red and black balls at node $i$ as well as the red and black balls of each of its neighbouring nodes. This concept is illustrated in Figure~\ref{fig:super_urn}. For each node $i\in V$, we use $\bR_i=\sum_{j \in \N_i'}R_j$ and $\bB_i=\sum_{j \in \N_i'}B_j$ to denote the number of red and black balls respectively in node $i$'s super urn. The total number of balls in the $i$th super urn is given by $\bT_i = \bR_i+\bB_i$. Similar to the classical Polya process, a draw is conducted for each node at each time, and then a number of balls of the same colour is added to that node's individual urn. In this case, however, the draw is conducted on the super urn of each node. As well, we may allow for the number of added balls to vary based on which node the draw was for, the colour of the draw, and the time at which the draw occurred. Thus, for node $i$ at time $t$, if a red ball is drawn we add $\Delta_{r,i}(t)$ red balls to node $i$'s individual urn and if a black ball is drawn we add $\Delta_{b,i}(t)$ black balls to node $i$'s individual urn. We use an indicator, $Z_{i,n}$ to denote the colour of the $n$th draw for node $i$, where $Z_{i,n}=1$ if the draw was red and $Z_{i,n}=0$ if the draw was black. We refer to  $\{\Delta_{b,i}(n)\}_{n=1}^{\infty}$ as the curing parameters and  $\{\Delta_{r,i}(n)\}_{n=1}^{\infty}$ as the infection parameters for node $i$.

{
\begin{figure}[!ht]
\centering 
	\includegraphics[width=0.7\linewidth]{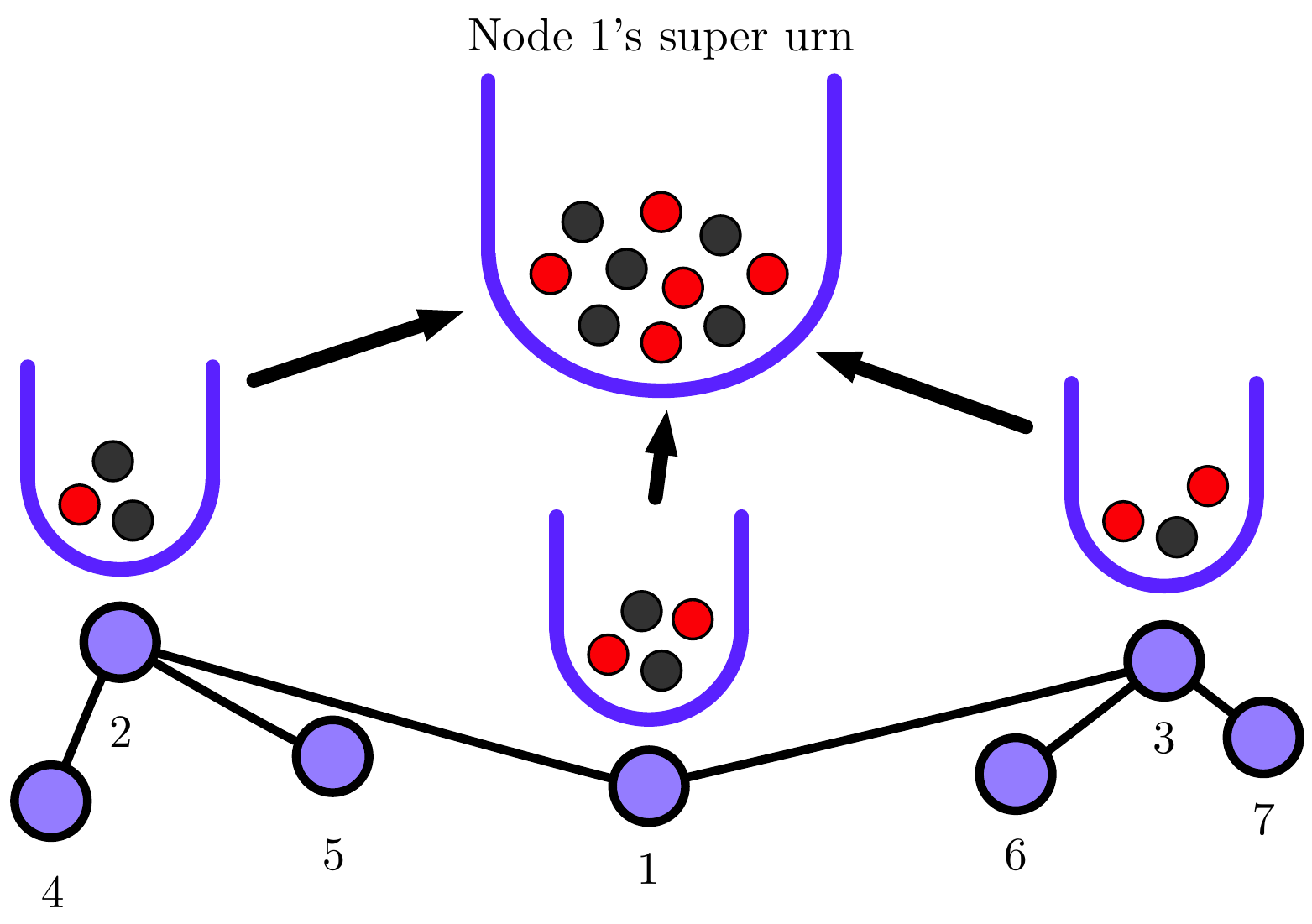}
	\caption{Illustration of a super urn in a network.}
	\label{fig:super_urn}
\end{figure}
}

For this process, we can represent the proportion of red balls in node $i$ after the $n$th draw by $U_{i,n}$, where 
\begin{align*}
	U_{i,n} &= \frac{R_i + \tsum \Delta_{r,i}(t)Z_{i,t}}{X_{i,n}}.
\end{align*}
Here
\begin{align} \label{eq:X_j,n}
X_{i,n}=T_i + \sum_{t=1}^{n} \Delta_{r,i}(t)Z_{i,t} + \Delta_{b,i}(t)(1-Z_{i,t})
\end{align}
represents the total number of balls in node $i$'s urn after the $n$th draw. We represent the proportion of red balls in node $i$'s super urn after the $n$th draw as $S_{i,n}$, where
\begin{align}\label{eq:S_n}
	S_{i,n} 
	&= \frac{\bR_i + \sum_{j \in \N_i'}\sum_{t=1}^{n} \Delta_{r,j}(t)Z_{j,t}}{\sum_{j \in \N_i'}X_{j,n}}.
\end{align}
We take $S_{i,0}=\frac{\bR_i}{\bT_i}$, which is just the initial super urn proportion of red balls for node $i$. We define
\begin{align*}
	\tS_n &= \frac{1}{N}\sum_{i=1}^N S_{i,n},
\end{align*}
which we refer to as the \textit{network exposure}.

The conditional probability of drawing a red ball for node $i$ at time $n$ , given the past history of draws for all the nodes 
\begin{equation*}
	\{Z_j^{n-1}\}_{j=1}^N := \{(Z_{1,1}, \ldots, Z_{1,n-1}),\ldots(Z_{N,1}, \ldots, Z_{N,n-1})\},
\end{equation*}
is given by
\begin{align*}
	P\left(Z_{i,n} = 1 | \{Z_j^{n-1}\}_{j=1}^N\right) &= \frac{\bR_i + \sum_{j \in \N_i'}\sum_{t=1}^{n-1} \Delta_{r,j}(t)Z_{j,t}}{\sum_{j \in \N_i'}X_{j,n-1}} \\
	&= S_{i,n-1}.
\end{align*}
Finally, we refer to
\begin{align*}
\avginf &= \frac{1}{N} \Nsum P(Z_{i,n}=1)
\end{align*}
as the \textit{average infection rate}.

\section{Game-Theoretic Setup}\label{sec:game}

Previously,~\cite{MH-FA-BG:18-2} has looked at the problem of minimizing the limiting average infection rate subject to a budget $\budget$ on the curing parameters $\{\Delta_{b,i}(n)\}_{i=1}^N$ at each time step. In this setup, $\budget$ is fixed for each time step and all other values are assumed to be given. We now consider a two player game, where player one looks to minimize the average infection rate, while player two looks to maximize this same value. Player one controls the distribution of curing parameters in accordance with a budget, $\budgetB$, while player two controls the distribution of the infection parameters  in accordance with a budget, $\budgetR$. In both cases the budget is fixed for all time steps, and for any given time step the allocation of the budget is performed prior to the draw. Thus, if either player allocates resources to a node for which the opposite colour ball is drawn, those resources will go to waste. 

Formally, player one's objective is as follows:
\begin{align}\label{eq:obj_curing}
	\minimize_{\{\{\Delta_{b,i}(k)\}_{i=1}^N:\sum_{i=1}^N\Delta_{b,i}(k)=\budgetB\},k=1,\ldots,t}\tI_t,
\end{align}
assuming the minimum exists, while player two's objective takes the form:
\begin{align}\label{eq:obj_infect}
	\maximize_{\{\{\Delta_{r,i}(k)\}_{i=1}^N:\sum_{i=1}^N\Delta_{r,i}(k)=\budgetR\},k=1,\ldots,t}\tI_t,
\end{align}
assuming the maximum exists.

For a general network, finding an optimal control policy for either player can be complicated, and the behaviour of $\avginf$ is difficult to study. As discussed in~\cite{MH-FA-BG:18-2}, there is a strong correlation between the behaviour of $\tS_n$ and $\avginf$, and thus we can simplify our problem by first considering a game on the expected network exposure for a given time step, rather than the limiting average infection rate.

We consider the expected network exposure $E[\tS_n|\mathcal{F}_{n-1}]$ as a two player zero-sum game, where player one looks to minimize the expected network exposure over the curing parameters $\{\Delta_{b,i}(n)\}_{i=1}^N$ and player two looks to maximize the same value over the infection parameters $\{\Delta_{r,i}(n)\}_{i=1}^N$. We begin with an important result about the expected network exposure:

\begin{proposition}\longthmtitle{Convexity-Concavity of Network Exposure}\label{thm:convex}
Let $\G = (V,\E)$ be a general network, and consider the Polya network contagion model on $\G$, with arbitrary initial conditions. Then, the expected network exposure $E[\tS_n|\F_{n-1}]$ is convex with respect to the curing parameters $\{\Delta_{b,i}(n)\}_{i=1}^N$ and concave with respect to the infection parameters $\{\Delta_{r,i}(n)\}^N_{i=1}$, for all $n$.
\end{proposition}
\begin{proof} 

Using~\eqref{eq:X_j,n} and~\eqref{eq:S_n}, we consider $E[\tS_n|\F_{n-1}]$ as a function of the vectors 
\begin{gather*}
	\bx=(x_1,\ldots,x_N)^T=(\Delta_{b,1}(n),\ldots, \Delta_{b,N}(n))^T\\
	\by=(y_1,\ldots,y_N)^T=(\Delta_{r,1}(n),\ldots, \Delta_{r,N}(n))^T, 
\end{gather*}
by reformulating~\eqref{eq:S_n} as follows:
\begin{align*}
	S_{i,n}=f_{i,n}(\bx,\by,Z_n) = \frac{c_i+\delta_i(\by,Z_n)}{c_i+d_i+\sigma_i(\bx,Z_n)+\delta_i(\by,Z_n)},
\end{align*}
where
\begin{gather*}
	c_i=\bR_i+\sum_{t=1}^{n-1}\sum_{j\in\N_i'}\Delta_{r,j}(t)Z_{j,t} \\
	d_i=\bB_i+\sum_{t=1}^{n-1}\sum_{j\in\N_i'}\Delta_{b,j}(t)(1-Z_{j,t}) \\
	\delta_i(\by,Z_n)=\sum_{j\in\N_i'}y_j Z_{j,n} \\
	\sigma_i(\bx,Z_n)=\sum_{j\in\N_i'}x_j(1-Z_{j,n}).
\end{gather*}

Alternatively, let $A$ represent the adjacency matrix (including self connections) for our given network $\G$, where $A_{ij}=1$ iff $i\in\N_j'$, and $A_{ij}=0$, otherwise. Then construct an $N\times N$ square matrix, $D$, where $D_{ij}=A_{ij}(1-Z_{j,n})$. Letting $D_i$ represent the $i^{th}$ row of the matrix $D$, we get $\sigma_i(x,Z_n) = D_i\bx$, where
\begin{align*}
	D_i\bx = \begin{bmatrix}A_{i1}(1-Z_{1,n})& \cdots &A_{iN}(1-Z_{N,n})\end{bmatrix}\begin{bmatrix}x_1\\ \vdots \\ x_N\end{bmatrix}.
\end{align*}

Likewise, we can construct an $N\times N$ square matrix, $C$, where $C_{ij}=A_{ij}Z_{j,n}$. Letting $C_i$ represent the $i^{th}$ row of the matrix $C$, we get the following:
\begin{align*}
	\delta_i(y,Z_n) = C_i\by = \begin{bmatrix}A_{i1}Z_{1,n}& \cdots &A_{iN}Z_{N,n}\end{bmatrix}\begin{bmatrix}y_1\\ \vdots \\ y_N\end{bmatrix}.
\end{align*}

Allowing us to alternatively write
\begin{align*}
	f_{i,n}(\bx,\by,Z_n) = \frac{c_i+C_i\by}{c_i+d_i+C_i\by+D_i\bx},
\end{align*}
noting of course that the matrices $C$ and $D$ are functions of $Z_n$.

Taking the expectation of $S_{i,n}$ given the history of the contagion process up to time $n-1$ we get
\begin{align}\label{eq:convex}
	E[S_{i,n}|\F_{n-1}]&=E[f_{i,n}(\bx,\by,Z_n)|\F_{n-1}] \nonumber\\
	&=\sum_{z_n\in\{0,1\}^N}f_{i,n}(\bx,\by,z_n)P(Z_n=z_n|Z^{n-1}).
\end{align}

We can note that $P(Z_n=z_n|Z^{n-1})$ is independent of  our choice of $\bx$, and for any fixed realization $Z_n = (z_{1,n},\ldots,z_{N,n})\in\{0,1\}^N$, $f_{i,n}(\bx,\by,z_n)$ is convex in $x$ over $\realnonnegative^N$. Hence, it follows that $E[S_{i,n}|\F_{n-1}]$ is convex in $\bx$ over $\realnonnegative^N$, and thus so too is $E[\tS_n|\F_{n-1}]$. Concavity in $\by$ follows from a symmetry argument wherein we note 
\begin{align*}
	1-f_{i,n}(\bx,\by,Z_n)= \frac{d_i+D_i\bx}{c_i+d_i+C_i\by+D_i\bx}
\end{align*}
is convex in $\by$, thus showing $f_{i,n}(\bx,\by,Z_n)$ is concave in $\by$. The rest of the proof follows as in the convexity argument for $\bx$.
\end{proof}

The natural symmetry of this model makes this result rather intuitive, and allows for us to establish a nice setup for the game theoretic problem. We have proven that $E[\tS_n|\F_{n-1}]$ is convex in $\bx$ and concave in $\by$ over $\realnonnegative^N\times\realnonnegative^N$. In order to get a better understanding of this function, we consider the partial derivatives with respect to $\bx$ and $\by$. We note that

\begin{gather*}
	\nabla_{\bx}f_{i,n}(\bx,\by,z_n)=\frac{-\nabla_{\bx}D_i\bx(c_i+C_i\by)}{(c_i+d_i+C_i\by+D_i\bx)^2}\\
	\nabla_{\by}f_{i,n}(\bx,\by,z_n)=\frac{\nabla_{\by}C_i\by(d_i+D_i\bx)}{(c_i+d_i+C_i\by+D_i\bx)^2},
\end{gather*}

and furthermore

\begin{gather*}
	\frac{\partial}{\partial x_j}D_i\bx=A_{ij}(1-z_{j,n})\\
	\frac{\partial}{\partial y_j}C_i\by=A_{ij}z_{j,n}.
\end{gather*}

Since $A_{ij}$ and $z_{j,n}$ can only take values in $\{0,1\}$, we get that $\frac{\partial}{\partial x_j}f_{i,n}(\bx,\by,z_n)\leq0$ and $\frac{\partial}{\partial y_j}f_{i,n}(\bx,\by,z_n)\geq0$. Using this fact in conjunction with equation~\eqref{eq:convex} we determine that over the space $\X\times\Y=\mathbb{R}^N_{\geq0}\times\mathbb{R}^N_{\geq0}$, $E[\tS_n|\F_{n-1}]$ has no saddle point. This lack of saddle point is demonstrated in Figure~\ref{fig:samplegame}, which depicts the shape of $E[\tS_n|\F_{n-1}]$ for a single node network. However, given our fixed allocation budgets $\budgetB$ and $\budgetR$, we restrict ourselves to considering sets of the form $\X\times\Y = \{\{\Delta_{b,i}(n)\}_{i=1}^N\in\realnonnegative^N|\sum_{i=1}^N\Delta_{b,i}(n)\leq\budgetB\}\times\{\{\Delta_{r,i}(n)\}_{i=1}^N\in\realnonnegative^N|\sum_{i=1}^N\Delta_{r,i}(n)\leq\budgetR\}$.

\begin{figure} [!ht]
\centering{
  \includegraphics[width=\linewidth]{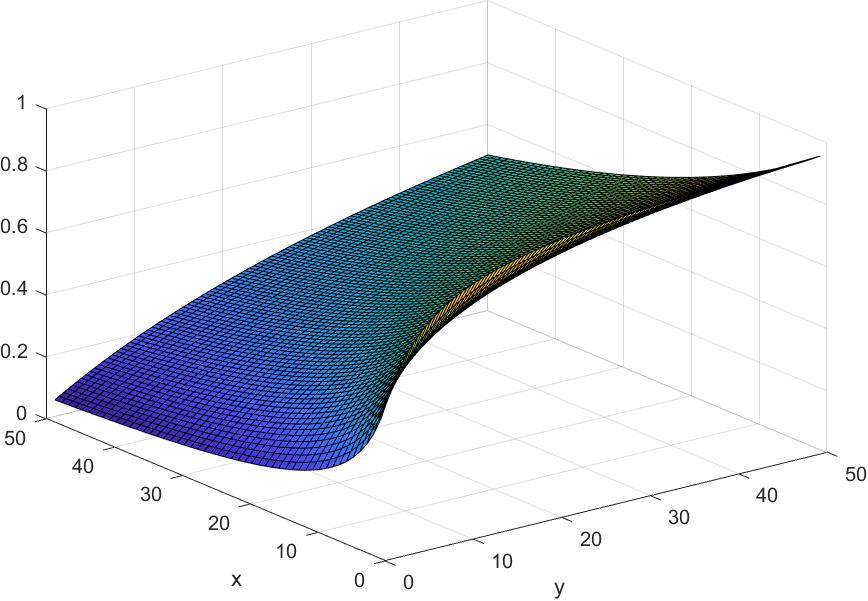}
}
  \caption{Sample plot of $E[\tS_n|\F_{n-1}]$ for a single node network. A saddle point will not exist over $\mathbb{R}^N_{\geq0}\times\mathbb{R}^N_{\geq0}$ regardless of the network structure.}
  \label{fig:samplegame}
\end{figure}
Returning to our game over the expected network exposure, we can then note that for any given $n$, the sets $\X = \{\{\Delta_{b,i}(n)\}_{i=1}^N\in\realnonnegative^N|\sum_{i=1}^N\Delta_{b,i}(n)\leq\budgetB\}$ and $\Y = \{\{\Delta_{r,i}(n)\}_{i=1}^N\in\realnonnegative^N|\sum_{i=1}^N\Delta_{r,i}(n)\leq\budgetR\}$ are convex and compact. This gives rise to the following result:

\begin{theorem}\longthmtitle{Nash Equilibrium for Network Exposure}\label{thm:nash_exposure}
Let $\G=(V,\E)$ be a general network equipped with the Polya network contagion model, with arbitrary initial conditions. For a given time $n$, consider a two-player zero-sum game where player one tries to minimize $E[\tS_n|\F_{n-1}]$ over the parameters $\{\Delta_{b,i}(n)\}_{i=1}^N$ and player two tries to maximize $E[\tS_n|\F_{n-1}]$ over the parameters $\{\Delta_{r,i}(n)\}_{i=1}^N$. Then, if we take our set of allowable policies to be of the form $\X\times\Y = \{\{\Delta_{b,i}(n)\}_{i=1}^N\in\realnonnegative^N|\sum_{i=1}^N\Delta_{b,i}(n)\leq\budgetB\}\times\{\{\Delta_{r,i}(n)\}_{i=1}^N\in\realnonnegative^N|\sum_{i=1}^N\Delta_{r,i}(n)\leq\budgetR\}$, the resulting game admits a Nash equilibrium. Moreover, the equilibrium policy will satisfy $\sum_{i=1}^N\Delta_{b,i}(n)=\budgetB$ and $\sum_{i=1}^N\Delta_{r,i}(n)=\budgetR$.
\end{theorem}
\begin{proof}
Since the function is convex-concave and over a compact set, the existence follows from the classical minimax theorem, see~\cite{DPB-AN-AEO:03}. By the definition of $E[\tS_n|\F_{n-1}]$, and since the function has no saddle point in the interior of its domain, the optimal policy will utilize the full budget.
\end{proof}

The equilibrium policy from Theorem~\ref{thm:nash_exposure} can be determined numerically using gradient descent algorithms~\cite{DPB:99} (as we are optimizing a convex/concave function over a simplex), but for large networks such algorithms can be computationally expensive. 

\section{Simulations}\label{sec:simulations}

We have proven that there exists a Nash equilibrium when we consider the two player infection-curing game for the expected network exposure $E[\tS_n|\F_{n-1}]$. This however, does not guarantee that the same holds true when considering a similar setup for the average infection rate as in~\eqref{eq:obj_curing} and~\eqref{eq:obj_infect}. We need to check first of all, that we can find policies for each player that mimic the behaviour of a hypothetical equilibrium policy. As well, we want to see if the optimal policies for the network exposure provide a good estimation for our hypothetical equilibrium policy. 

We consider three basic networks, as depicted in Figures~\ref{subfig:linenet},~\ref{subfig:starnet}, and~\ref{subfig:circlenet}. Each network is assigned a uniform initial distribution for both red and black balls. A fixed budget of $\budgetB=\budgetR$ is given for both players at each time step. At any given time, the optimal policy for the infection-curing game over the expected network exposure is calculated for both the curing and infection parameters. The first set of simulations uses this optimal policy for both players. In the second set of simulations only the curing parameters are assigned using this optimal policy, while the infection parameters are distributed uniformly. In the final set of simulations, the infection parameteres are assigned using the optimal policy, while the curing parameters are uniformly distributed. At each time $n$, the draw values $\{Z_{i,n}\}_{i=1}^N$ are averaged over 1000 trials, and then averaged over all nodes to obtain the empirical average infection rate $\avginf$. The plots in Figure~\ref{fig:infectiongame} display the resulting output for each network plotted against time.

For all three networks, we observe that case 3 generally leads to the highest average infection rate and case 2 results in the lowest average infection rate. While we do not know what the optimal policy for either player would be in the case of $\avginf$, we can see that the behaviour observed suggests that a Nash equilibrium exists. This is less obvious when we consider the network in Figure~\ref{subfig:circlenet}. For this network, it makes sense that a hypothetically optimal policy would be close to uniform, and so we observe less distinct separation between each of the cases in Figure~\ref{subfig:circle} as compared to Figure~\ref{subfig:line} and Figure~\ref{subfig:star}. The results in Figure~\ref{subfig:circle} do not, however, preclude the possibility that there exists an equilibrium policy for both players.

In general, the results support the notion that a Nash equilibrium exists for the game over the average network infection, $\avginf$, and that using the solutions obtained from the gradient descent over the expected network exposure may provide a good estimation as to what such an equilibrium policy may look like. However, the desire to guarantee such claims necessitates further follow up, as these simulations provide only a narrow view.

\begin{figure*}[!th]
\centering
	\subfigure[Seven node line network.]{ \includegraphics[width=0.25\linewidth]{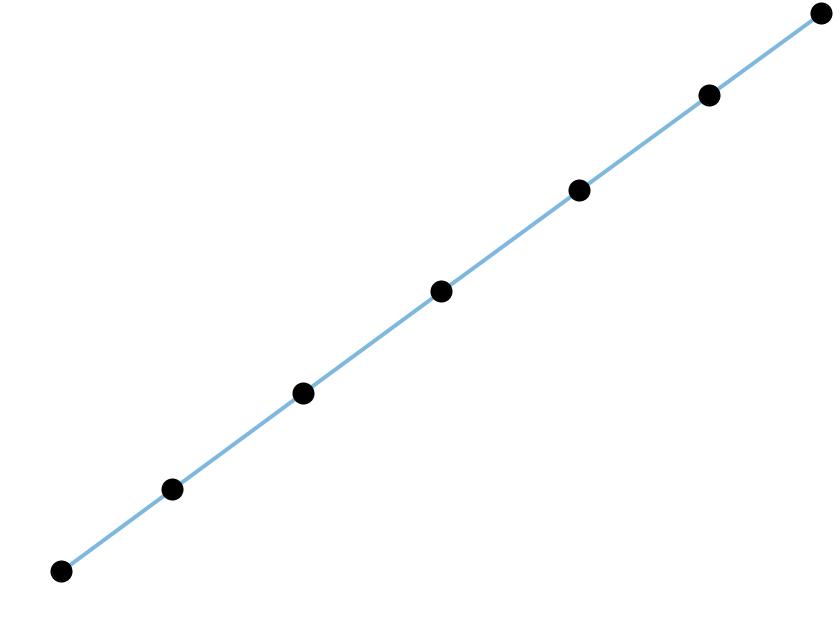} \label{subfig:linenet}}
\qquad
{
	\subfigure[Six node star network.]{ \includegraphics[width=0.25\linewidth]{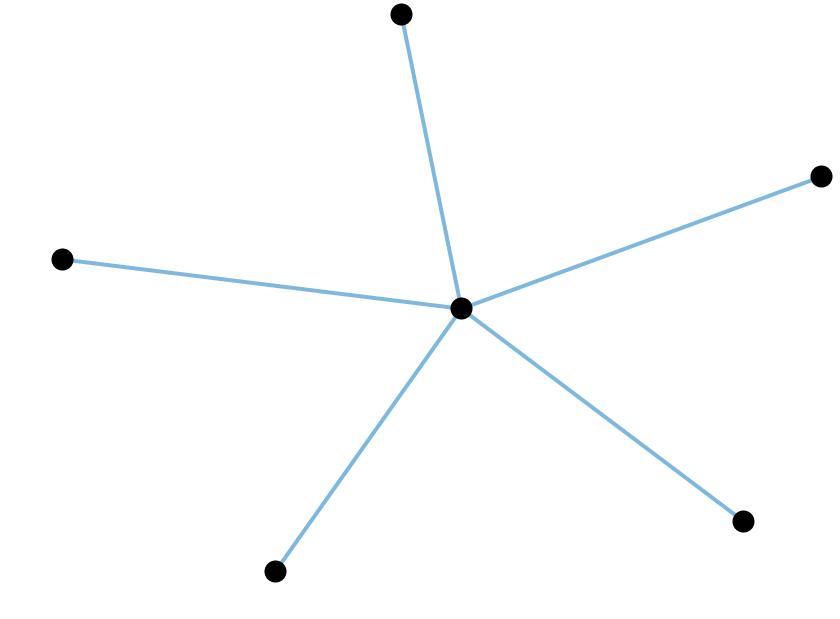} \label{subfig:starnet}}
}
\qquad
{
	\subfigure[Six node circular network.]{ \includegraphics[width=0.25\linewidth]{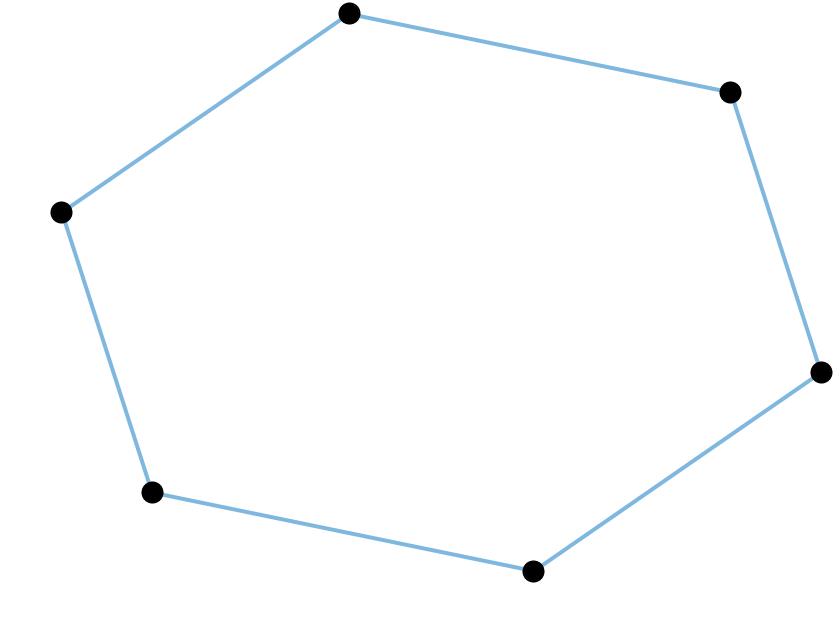} \label{subfig:circlenet}}
}
\\
{
	\subfigure[Plot of empirical average infection rate $ \avginf $ for line network depicted in Figure~\ref{subfig:linenet}.]{ \includegraphics[width=0.45\linewidth]{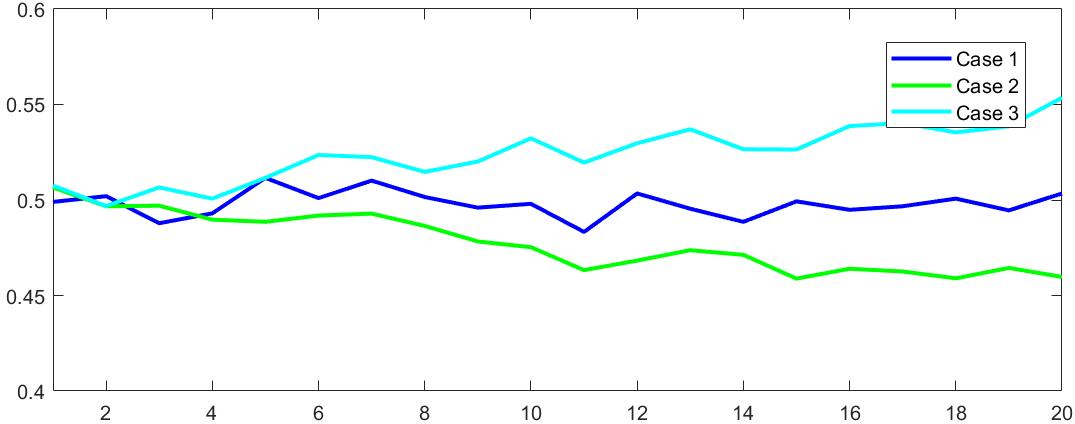} \label{subfig:line}}
}
\qquad
{
	\subfigure[Plot of empirical average infection rate $ \avginf $ for star network depicted in Figure~\ref{subfig:starnet}.]{ \includegraphics[width=0.45\linewidth]{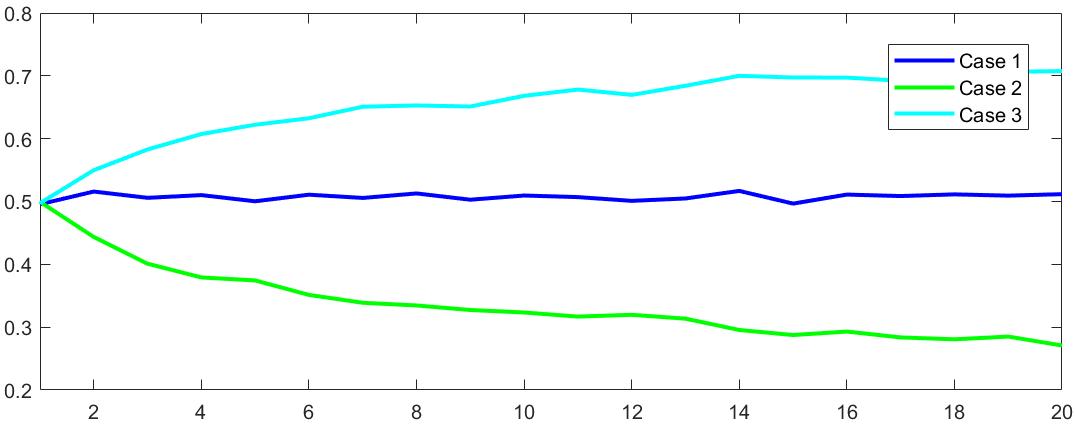} \label{subfig:star}}
}
\\
{
	\subfigure[Plot of empirical average infection rate $ \avginf $ for circular network depicted in Figure~\ref{subfig:circlenet}.]{ \includegraphics[width=0.45\linewidth]{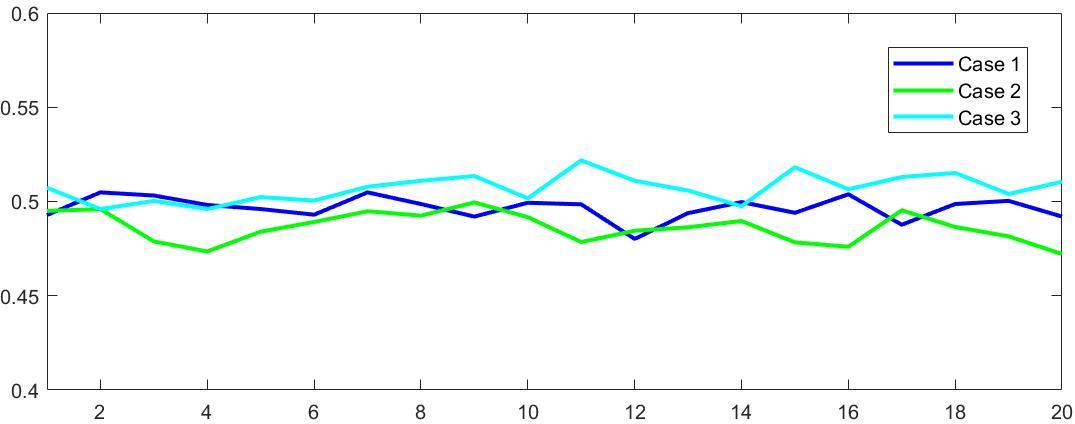} \label{subfig:circle}}
}
\caption{Comparison of simulation results for three test networks. Three cases were compared over each network. Case one used the equilibrium policies for the infection-curing game on the expected network exposure. These policies were obtained through gradient descent algorithms~\cite{DPB:99}. Case two used a uniform distribution for only the infection parameters. Case three used a uniform distribution for only the curing parameters. All simulations were performed with a uniform initial distribution of 10 red and black balls per node. A budget of $\budgetR=\budgetB=10N$ was used in all cases, where $N$ is the number of nodes in the network.}
\label{fig:infectiongame}
\end{figure*}

\section{Conclusions and Future Directions}\label{sec:future}

In this paper, we explored two player zero-sum games over Polya contagion networks, verifying analytically the existence of a Nash equilibrium for an infection-curing game on the expected network exposure, and providing empirical evidence to support the existence of an equilbrium point for the infection-curing game on the average network infection. Future work may further investigate the existence of such an equilibrium point. Considering variations of this problem may also be of interest; for example, it may be useful to explore the infection-curing game under the restriction that only a single node be provided resources, or to consider the case where the underlying ball distribution is unknown to the participating players. Such problems could be considered for both the expected network exposure and average infection rates. The infection-curing game on the average infection rate can also be further broken down into exploring both finite and infinite horizon cases.

\bibliographystyle{ieeetr}
\bibliography{alias,Main-add,GH-add}
\end{document}